\documentclass[12pt,a4paper]{amsart}
\usepackage[latin1]{inputenc}
\usepackage[english]{babel}
\usepackage{amsmath,amssymb,amsthm}

\usepackage{enumerate}
\usepackage{cite}

\theoremstyle{plain}
\newtheorem{theorem}{Theorem}[section]
\newtheorem{lemma}[theorem]{Lemma}
\newtheorem{corollary}[theorem]{Corollary}
\newtheorem{prop}[theorem]{Proposition}

\theoremstyle{definition}
\newtheorem{definition}[theorem]{Definition}
\newtheorem{example}[theorem]{Example}

\begin{document}

\title{A note on the luc-compactification of locally compact groups}



\author[S. Jafar-Zadeh]{Safoura Jafar-Zadeh}

\address{Department of Mathematics,
342 Machray Hall, 186 Dysart Road,
University of Manitoba, Winnipeg, MB R3T 2N2, Canada.
} 
\email{\texttt{jsafoora@gmail.com}}

\begin{abstract}
 Let $G$ be a locally compact non-compact topological group. We show that $G^{luc}$ ($G^*$) is an F-space if and only if $G$ is a discrete group.
\end{abstract}

\maketitle

\begin{section}{Introduction}

 Let $G$ be a Hausdorff locally compact group and $C_b(G)$ denote the Banach space of bounded continuous functions on $G$ with the uniform norm. An element $f$ in $C_b(G)$ is called left uniformly continuous if $g\mapsto l_gf$, where $l_gf(x)=f(gx)$ is a continuous map from $G$ to $C_b(G)$. Let $LUC(G)$ denote the Banach space of left uniformly continuous functions with the uniform norm. It can be seen that $C_c(G)$ the space of continuous compactly supported functions on $G$ is a subspace of $LUC(G)$. The dual space of $LUC(G)$, $LUC(G)^*$, with the following Arens-type product forms a Banach algebra:
$$<m.n,f>=<m,nf>\ \ \ and \ \ \  < nf,g> = <n,l_g(f)>,$$

where $m,n\in LUC(G)^*$, $f\in LUC(G)$ and $g\in G.$ In fact this product is an extension to $LUC(G)^*$ of the convolution product on $M(G)$.\\

The $luc$-compactification of a locally compact group $G$, denoted by $G^{luc}$ can be regarded as a generalization to locally compact groups, of the familiar Stone-Cech compactification for discrete groups. As a topological space, $G^{luc}$ is the Gelfand spectrum of the unital commutative $C^*-$algebra of left uniformly continuous functions on $G$, i.e., $LUC(G)=C(G^{luc})$. The restriction of the product on $LUC(G)^*$ to the subset $G^{luc}$ of $LUC(G)^*$ turns $G^{luc}$ into a semigroup. This product in fact is an extension to $G^{luc}$ of the product on $G$. In case $G$ is a discrete group, since $LUC(G)=l^{\infty}(G)=C(\beta G)$, the $luc-$compactification of $G$ is just its Stone-cech compactification. The corona of the $luc-$compactification of $G$, $G^{luc}\setminus G$, is demoted by $G^*$ and is a closed subsemigroup of the compact semigroup $G^{luc}$.\\

F-spaces were studied in detail by L. Gillman and M. Henriksen in 1956 \cite{MR0078980} as the class of spaces for which $C(X)$ is a ring in which every finitely generated ideal is a principal ideal. Several conditions both topological and algebraic were proved equivalent for an F-space (see  \cite{MR0407579}).
\begin{definition}
A completely regular space $X$ is an F-space if and only if for any continuous bounded function $f$ on $X$ there is a continuous bounded function $k$ on $X$ such that $f=k |f|.$
\end{definition}
Many of the proofs in the case of discrete groups use the fact that for a discrete space the Stone-Cech compactification and its corona are an F-space. This is specially useful due to the following lemma. A proof is given in \cite{MR1297310} lemma 1.1.

\begin{lemma}\label{F-space}
If $X$ is a compact space then $X$ is an F-space if and only if for sigma-compact subsets $A$ and $B$ of $X$, $\bar{A}\cap B=\emptyset$ and $A\cap\bar{B}=\emptyset$ implies that $\bar{A}\cap\bar{B}\neq\emptyset.$
\end{lemma}

It follows from lemma \ref{F-space} that:
\begin{lemma}
If $X$ is a discrete space, then every compact subset of $\beta X$ is an F- space. In particular $\beta X\setminus X$ is an F- space.
\end{lemma}
We observed that for a locally compact group the $luc-$compactification (it's corona) is an F-space if and only if the group is discrete.

\begin{definition}
A topological space $(X,\tau)$ is called metrizable if there is a metric $d:X\times X\to [0,\infty)$ such that the topology induced by $d$ is $\tau$.
\end{definition}
Any metrizable space is Hausdorff and first countable. Being Hausdorff and first countable is not always enough for a topological space to be metrizable. However, in case we are dealing with topological groups being Hausdorff and first countable turns out to be sufficient to be metrizable. This is the content of the Birkhoff-Kakutani theorem (see \cite{MR551496} theorem 8.5).
\begin{prop}
(Birkhoff-Kakutani) For a topological group $G$ the following three properties are equivalent:
\begin{enumerate}[(i)]
\item $G$ is first countable;\\
\item $G$ is metrizable;\\
\item there exists a left-invariant compatible metric on $G$.
\end{enumerate}
\end{prop}

\end{section}
\begin{section}{Proof of the main result} Our main goal is to show that if $G$ is a locally compact non-discrete group then neither $G^{luc}$ nor $G^*$ are F-spaces. We first observe that in general a locally compact non-discrete topological space that contains a convergent sequence is not an F-space. To see this we recall the following lemma.

\begin{lemma}\label{topol}
Suppose that $X$ is a locally compact space and $K$ is a compact set in $X$ and  let $D$ be an open subset with $K\subseteq D$. Then there exists an open set $E$ with $\bar{E}$ compact such that $K\subseteq E\subseteq \bar{E}\subseteq D.$
\end{lemma}

Note that  the Urysohn's lemma fails for $luc-$functions. An example is given below. 
\begin{example}
Consider the closed sets $\{(x,\frac{1}{x}),x\in\mathbb{R}\}$ and $\{(x,0),x\in\mathbb{R}\}$ in $\mathbb{R}^2$. It can be easily checked that there is no uniformly continuous function separating these closed sets.
\end{example}

\begin{theorem}
If $X$ is a Hausdorff locally compact non-discrete topological space (group)  with a nontrivial convergent sequence then $X$ is not an F-space.
\end{theorem}
\begin{proof}
Suppose that $X$ contains a non-trivial sequence $(x_n)$ convergent to a (non-isolated) point  $x_0$ in $X$. Since $x_0$ is not isolated and $G$ is locally compact we can construct inductively a nested family $\{K_n\}_{n\geq 2}$ of compact neighbourhoods around $x_0$ such that $x_n\notin K_{n+1}$ for each $n$. To see this note that since $x_1\neq x_0$ there is a precompact open set $U_0$ such that $x_0\in U_0$ and $x_1\not\in U_0$. Apply lemma \ref{topol} with $K:=\{x_0\}$ and $D:=U_0$ to get $K_2$. Without loss of generality we can assume that $(x_n)_{n\geq 2}\in K_2$. Suppose that the compact set $K_n$ is given such that $x_{n-1}\not\in K_n$. Since $x_n\neq x_0$ there is a pre compact open set $U_n$ such that $x_0\in U_n\subseteq K_n$ and $x_n\not\in U_n$.  Apply lemma \ref{topol} to get $K_{n+1}$ and note that $x_n\not\in K_{n+1}$.\\
 Consider a convergent series $\sum_{n=1}^{\infty}a_n$, where $\sum_{n=1}^m a_n\sum_{n=1}^{m+1} a_{n}<0$ that is not absolutely convergent. Using Urysohn's lemma define the compactly supported function $f_n$ such that  $f_n(X\setminus K_n)=0$ and $f_n(K_{n-1})=a_{n-1}$. Let $f=\sum f_n$. Then $f$ is a continuous (compactly supported) function on $X$. For any function $k$ where $f=k |f|$, $k$ is not continuos at $x_0$ since the series we chose is not absolutely convergent..  
\end{proof}

So to show that $G^{luc}$ ($G^*$) is not an F-space, it is enough to show that it contains a nontrivial convergent sequence. In case that $G$ is a metrizable non-discrete locally compact group, $G$ has a nontrivial convergent sequence, and we can use this sequence to build one in $G^{luc}$. But what if $G$ is not metrizable? The following theorem is due to Kakutani-Kodaira (see \cite{MR551496} theorem 8.7).

\begin{prop}
(Kakutani-Kodaira) Let $G$ be a $\sigma$-compact locally compact group. For any sequence $(U_n)_{n\geq 0}$ of neighbourhood of identity of $G$, there exists a compact normal subgroup $N$ of $G$ contained in $\cap_{n\geq 0}U_n$ such that $G/K$ is metrizable.
\end{prop}
We have the following hereditary property (see \cite{MR551496} theorem 5.38 e):
\begin{lemma}
Let $G$ be a topological group and $H$ be a closed subgroup then if $H$ and $G/H$ are first countable so is $G$.
\end{lemma}

In a $\sigma-$compact locally compact group $G$ the existence of a non-metrizable compact normal subgroup is the only drawback to the non-metrizability of $G$.
\begin{prop}
Let $G$ be a non-metrizable $\sigma-$compact locally compact group and $N$ a compact normal subgroup of $G$ such that $G/N$ is metrizable. Then $N$ is non-metrizable.
\end{prop}
\begin{proof}
This follows from proposition 1.5 and lemma 2.5.
\end{proof}
If $G$ is a locally compact group it contains a subgroup $H$ that is $\sigma$-compact and clopen. This is the content of the next lemma.
\begin{lemma}
If $G$ is a locally compact non-discrete group then any symmetric pre compact neighbourhood of the identity of $G$ contains a subgroup $H$ that is $\sigma$-compact clopen and non-discrete.
\end{lemma}

\begin{proof}
Let $U$ be a precompact symmetric open neighbourhood of identity and form the subgroup $H=\cup_{n\in\mathbb{N}}U^n.$ Then $H$ is clopen and $\sigma$- compact. $H$ is non-discrete since it is open in $G$ and $G$ is a non-discrete topological group.
\end{proof}

Consider the group $\mathbb{Z}_2=\{0,1\}$ taken with the discrete topology. Let $\kappa$ be the product space $\mathbb{Z}_2^{\kappa}$ with the product topology. Then the space $D^{\kappa}:=\mathbb{Z}_2^{\kappa}$ is compact totally disconnected and perfect. In particular, $D^{\aleph_0}$ is a metrizable space. In fact every compact locally disconnected perfect metrizable space is homeomorphic to $D^{\aleph_0}$.\\
In 1952 P.S. Alexandroff announced that every compact metric space is a continuous image of the cantor set $\{0,1\}^{\aleph_0}$.
In 1958 Kuzminov \cite{MR0104753} showed in fact that any compact group is dyadic, i.e., a continuous image of a Cantor cube. 
\begin{prop}
Any compact group is dyadic.
\end{prop}
The following is an independent result of Katetov and Efimov \cite{MR0152987}. 
\begin{prop}
Every non-isolated point of a dyadic space is the limit of a sequence of distinct points.
\end{prop}
\begin{lemma}
If $G$ is a locally compact non discrete group then $G$ has a non-trivial convergent sequence.
\end{lemma}
\begin{proof}
It is easy to see that if $G$ is a metrizable non-discrete locally compact group then $G$ has a non-trivial convergent sequence. If $G$ is a locally compact non-discrete group then $G$ has a subgroup $H$ that is sigma-compact, clopen and non-discrete. By Kakutani-Kodaira theorem there is a compact normal subgroup $N$ of $H$ such that $H/N$ is metrizable. If $H$ is metrizable then any point in $H$ is a limit point of a convergent sequence in $G$. Since $H$ is open in $G$ this sequence is also convergent in $G$. If $H$ is not metrizable then since $N/H$ is metrizable $N$ can't be metrizable otherwise $H$ would be metrizable. Since $N$ is not metrizable it can't be finite and so since it is compact it has a non-isolated point. By structure theorem of Kuzminov we know that $N$ is a continuous image of the Cantor cube $\{0,1\}^{\kappa}$ for some cardinal number $\kappa$. Now by Katetov and Efimov we have that every point in $N$ is the limit point of a non-trivial sequence. Any such nontrivial convergent sequence is also convergent in the open subgroup $H$ and therefore in $G$.
\end{proof} 

\begin{theorem}
If $G$ is a locally compact non-discrete group then neither $G^{luc}$ nor $G^*$ are F-spaces.
\end{theorem} 

\begin{proof}
Suppose that $G$ is a locally compact non-discrete group. Then $G$ contains a non-trivial sequence say $(x_n)$ convergent to $x$. Suppose that $z$ is a right cancellable element in $G^*$ then the nontrivial sequence $(x_nz)$ convergent to $xz$. So in both cases, they are not F-spaces.
\end{proof}

\begin{corollary}

If $G$ is a locally compact group and $H$ is a discrete group then if $G^*$ is homeomorphic to $H^*$, then $G$ must be discrete. 
\end{corollary}
\begin{proof}
Since $H$ is discrete $H^*$ is an F-space. Since $G^*$ is homeomorphic to $H^*$ it is also an F-space and therefore $G$ is discrete.
\end{proof}

\begin{corollary}
Let $G$ be a locally compact non-discrete group. Then $G^*$ is not discrete.
\end{corollary}
\begin{proof}
Since $G^*$ is not an $F-$space  the result follows.
\end{proof}

\begin{definition}
A point in a topological space is called a P-point if every $G_{\delta}-$set containing the point is a neighbourhood of the point.
\end{definition}

\begin{corollary}
Suppose that $G$ is a locally compact non-discrete group. Then $G^*$ doesn't contain any $P-$point.
\end{corollary}
\begin{proof}
For each point $p$ in $G^*$ there is a non-trivial convergent sequence $(x_n)$. Let $\{U_n\}$ be a family of open sets of $p$ such that for each $n$, $x_1,x_2,...,x_n\not\in U_n$. Then for each $n$, $x_n\not\in\cap_m U_m$ so $\cap_m U_m$ can't be open and so $p$ is not a $P-$point.
\end{proof}
If $G$ is a discrete group then under the continuum hypothesis the set of $P-$points in $G^*$ forms a dense subset in $G^*$ (see the remark on page 385 of \cite{MR1377702}).
\begin{corollary}
If $G$ is a locally compact group $G^{luc}\setminus G$ has a $P-$point if and only if $G$ is discrete.
\end{corollary}

\end{section}

\nocite{*}
\bibliographystyle{plain}

\bibliography{mybib}

\def\cprime{$'$}
\begin{thebibliography}{1}

\bibitem{MR999922}
John~F. Berglund, Hugo~D. Junghenn, and Paul Milnes.
\newblock {\em Analysis on semigroups}.
\newblock John Wiley \& Sons Inc., New York, 1989.

\bibitem{MR0152987}
B.~Efimov.
\newblock On dyadic spaces.
\newblock {\em Dokl. Akad. Nauk SSSR}, 151:1021--1024, 1963.

\bibitem{MR1039321}
Ryszard Engelking.
\newblock {\em General topology}, volume~6 of {\em Sigma Series in Pure
  Mathematics}.
\newblock Heldermann Verlag, second edition, 1989.

\bibitem{MR1377702}
M.~Filali.
\newblock Right cancellation in {$\beta S$} and {$UG$}.
\newblock {\em Semigroup Forum}, 52(3):381--388, 1996.

\bibitem{MR0078980}
Leonard Gillman and Melvin Henriksen.
\newblock Rings of continuous functions in which every finitely generated ideal
  is principal.
\newblock {\em Trans. Amer. Math. Soc.}, 82:366--391, 1956.

\bibitem{MR0407579}
Leonard Gillman and Meyer Jerison.
\newblock {\em Rings of continuous functions}.
\newblock Springer-Verlag, New York, 1976.

\bibitem{MR551496}
Edwin Hewitt and Kenneth~A. Ross.
\newblock {\em Abstract harmonic analysis. {V}ol. {I}}, volume 115.
\newblock Springer-Verlag, second edition, 1979.

\bibitem{MR1297310}
Neil Hindman and Dona Strauss.
\newblock Cancellation in the {S}tone-\v {C}ech compactification of a discrete
  semigroup.
\newblock {\em Proc. Edinburgh Math. Soc. (2)}, 37(3):379--397, 1994.

\bibitem{MR0104753}
V.~Kuz{\cprime}minov.
\newblock Alexandrov's hypothesis in the theory of topological groups.
\newblock {\em Dokl. Akad. Nauk SSSR}, 125:727--729, 1959.

\end{thebibliography}

\end{document}